\documentclass[10pt,leqno]{article} 
\usepackage{latexsym}
\usepackage{amssymb, color}
\usepackage{amsmath}
\usepackage{amsthm}
\usepackage{amsfonts}
\usepackage{mathrsfs}
\usepackage{graphics}
\newtheorem{thm}{Theorem}[section]
\newtheorem{lma}{Lemma}[section]

\newcommand{\beqa}{\begin{eqnarray}}
\newcommand{\eeqa}{\end{eqnarray}}

\newcommand{\beq}{\begin{equation}}
\newcommand{\eeq}{\end{equation}}
\newcommand{\lbl}{\label}
\newcommand{\s}{\; \;}

\newcommand{\la}{\lambda}

\newcommand{\ra}{\rightarrow}

\newcommand{\R}{\mathbb{R}}
\newcommand{\De}{\Delta}
\newcommand{\lt}{\left}
\newcommand{\rt}{\right}

\title{Bounded solutions for a class of Hamiltonian systems}

\author{
Philip Korman\thanks{Supported in part by the Taft faculty grant at the University of Cincinnati  }   \hspace{.1cm} and Guanying Peng \\ 
Department of Mathematical Sciences \\ 
University of Cincinnati \\ 
Cincinnati Ohio 45221-0025 \\
}

\date{}

\begin{document}

\maketitle
\begin{abstract} 
We obtain bounded for all $t$ solutions of ordinary differential equations as limits of the solutions of the corresponding Dirichlet problems on $(-L,L)$, with $L \ra \infty$. We derive  a priori estimates for the Dirichlet problems, allowing passage to the limit, via a diagonal sequence. This approach carries over to the PDE case.
 \end{abstract}

\begin{flushleft}
\end{flushleft}

\begin{flushleft}
\end{flushleft}

\section{Introduction}
\setcounter{equation}{0}
\setcounter{thm}{0}
\setcounter{lma}{0}

For $-\infty<t<\infty$, we consider the equation 
\beq
\lbl{i1}
u''-a(t)u^3=f(t) \,,
\eeq
with continuous functions $a(t)>0$ and $f(t)$. 
Clearly, ``most" solutions of (\ref{i1}) blow up in finite time, for both increasing and decreasing $t$. By using two-dimensional shooting, S.P. Hastings and J.B. McLeod \cite{H} showed that the equation (\ref{i1}) has a uniformly bounded on $(-\infty, \infty)$ solution, in case of constant $a(t)$ and uniformly bounded $f(t)$. Their proof used some non-trivial topological property of a plane. We use a continuation method and passage to the limit as in P. Korman and A.C. Lazer \cite{KL} to obtain the existence of a uniformly bounded on $(-\infty, \infty)$ solution for (\ref{i1}), and for similar systems. 
We  produce a bounded solution as a limit of the solutions of the corresponding Dirichlet problems
\beq
\lbl{i2}
u''-a(t)u^3=f(t) \s\s \mbox{for $t \in (-L,L)$}, \s u(-L)=u(L)=0\,,
\eeq
as $L \ra \infty$. If $f(t)$ is bounded, it follows by the maximum principle that the solution of (\ref{i2}) satisfies a uniform in $L$ a priori estimate, which allows passage to the limit.

Then we use a variational approach motivated by P. Korman and A.C. Lazer \cite{KL} (see also P. Korman, A.C. Lazer and Y. Li \cite{KLL}), to get a similar result for a class of Hamiltonian systems. Again, we consider the corresponding Dirichlet problem on $(-L,L)$, which we solve by the minimization of the corresponding functional, obtaining in the process a uniform in $L$ a priori estimate, which allows passage to the limit as $L \ra \infty$.
\medskip

We used a similar approach to obtain uniformly bounded solutions  for a class of PDE systems of Hamiltonian type. The challenge was to adapt the elliptic estimates in case only the $L^{\infty}$ bound is known for the right hand side.

\section{A model equation}
\setcounter{equation}{0}
\setcounter{thm}{0}
\setcounter{lma}{0}

\begin{thm}\lbl{thm:1}
Consider the equation (for $u=u(t)$)
\beq
\lbl{1}
u''-a(t)u^3=f(t) \,,
\eeq
where the given functions $a(t) \in C(\R)$ and $f(t) \in C(\R)$ are assumed to satisfy 
\begin{equation*}
%\lbl{2}
  |f(t)| \leq M,\s \mbox{  for all $t \in \R$, and some constant $M>0$} \,,
\end{equation*}
and
\begin{equation*}
%\lbl{3}
a_0\leq a(t) \leq a_1, \s \mbox{ for all $t \in \R$, and  some  constants $a_1\geq a_0>0$} \,.
\end{equation*}
Then the problem (\ref{1}) has a classical solution uniformly bounded for all $t \in \R$, i.e., $|u(t)| \leq K$ for  all $t \in \R$, and some $K>0$. Such a solution is unique.
\end{thm}

\begin{proof}
We shall obtain a bounded solution as a limit of solutions to the corresponding Dirichlet problems
\beq
\lbl{4}
u''-a(t)u^3=f(t) \s\s \mbox{for $t \in (-L,L)$}, \s u(-L)=u(L)=0\,,
\eeq
as $L \ra \infty$. To prove the existence of solutions, we embed (\ref{4}) into a family of problems 
\beq
\lbl{4a}
u''-\la a(t)u^3=f(t) \s\s \mbox{for $t \in (-L,L)$}, \s u(-L)=u(L)=0\,,
\eeq
with $0 \leq \la \leq 1$. The  solution at $\la =0$, and other $\la$, can be locally continued in $\la$ by the implicit function theorem, since the corresponding linearized problem 
\[
w''(t)-3\la a(t)u^2(t)w(t)=0 \s\s \mbox{for $t \in (-L,L)$}, \s w(-L)=w(L)=0 
\]
has only the trivial solution $w(t) \equiv 0$, as follows by the maximum principle.   Multiplying (\ref{4a}) by $u$ and integrating, we get a uniform in $\la$ bound on $H^1$
norm of the solution, which implies the bound in $C^2$ (using Sobolev's embedding and the equation (\ref{4a}); this bound depends on $L$). It follows that the continuation can be performed for all $0 \leq \la \leq 1$. At $\la=1$, we get the desired solution of (\ref{4}).

We claim that there is a \emph{uniform in $L$ bound} in $C^2[-L,L]$ for any solution of (\ref{4}), i.e.,  there is a constant $K>0$, so that for all $t \in [-L,L]$, and all $L>0$,
\beq
\lbl{5}
|u(t)| \leq K \,, \s |u'(t)| \leq K \,,\s \mbox{and } \s |u''(t)| \leq K  \,.
\eeq
Indeed, if $t_0$ is a point of positive maximum of $u(t)$, then from the equation (\ref{4}) we get
\[
-a_0 u^3(t_0) \geq f(t_0)  \geq -M \,,
\]
which gives us an upper  bound on $u(t_0)$. Arguing similarly at  a point of negative minimum of $u(t)$, we get a lower   bound on $u(t)$, and then conclude the first inequality in (\ref{5}). From the equation (\ref{4}) we get a uniform bound on $|u''(t)|$. Note that for all $t\in\R$, we can write
\begin{equation}\label{2.7}
u(t+1) =u(t) +u'(t)+ \int^{t+1}_t (t+1-\xi) u'' (\xi)\,d\xi,
\end{equation}
from which we immediately deduce a uniform bound on $|u'(t)|$.

We now take a sequence $L_j \rightarrow \infty\,$, and 
denote by $u_j(t) \in H_0^1(-\infty,\infty)$  the bounded solution of the problem (\ref{4}) on the interval $(-L_j,L_j)$, extended as zero to the outside of the interval $(-L_j,L_j)$. For all $t_1<t_2$, writing
\beq
\lbl{8}
 \lt|u_j(t_2) - u_j(t_1)\rt| = \lt|\int^{t_2}_{t_1} u'_j \,dt\rt|\, \leq \sqrt{t_2 - t_1} 
\left( \int^{t_2}_{t_1} \lt(u_j'\rt)^2 \, dt \, \right)^{1/2} 
\eeq
\[
 \leq K \left( t_2 - t_1 \right) \,,
\]
in view of   (\ref{5}), we conclude that the sequence $\{u_j(t)\}$ is equicontinuous and uniformly 
bounded on every interval $[-L_{p},L_{p}]\,$.  By the Arzela-Ascoli theorem, it has a uniformly 
convergent subsequence on every $[-L_{p},L_{p}]\,$. 
So let $\{u^1_{j_k}\}$ be a subsequence of $\{u_j\}$ that converges uniformly on 
$[-L_1,L_1]\,$.  Consider this subsequence on $[-L_2,L_2]$ and select a 
further subsequence $\{u^2_{j_k}\}$ of $\{u^1_{j_k}\}$ that converges 
uniformly on $[-L_2,L_2]\,$.  We repeat this procedure for all $m$, and then 
take the diagonal sequence $\{u^k_{j_k}\}\,$.  It follows that it 
converges uniformly on any bounded interval to a function $u(t)\,$.

Expressing $\left(u^k_{j_k}\right)''$ from the equation (\ref{4}), we conclude that the 
sequence $\lt\{\left(u^k_{j_k}\right)''\rt\}\,$, and then also $\lt\{\left(u^k_{j_k}\right)'\rt\}$ (in view of (\ref{2.7})), converge 
uniformly on bounded intervals.  Denote $v(t) :=\lim _{k \ra \infty} \left(u^k_{j_k}\right)''(t)$. For $t$ belonging to any bounded interval $(a,b)$, similarly to (\ref{2.7}), we write 
\[
u^k_{j_k}(t) =u^k_{j_k}(a) +(t-a)\left(u^k_{j_k}\right)'(a)+ \int^t_a (t-\xi) \left(u^k_{j_k}\right)'' (\xi)\,d\xi \,, 
\]
and  conclude that $u(t) \in C^2(-\infty,\infty)$, and $u''(t)=v(t)$.  Hence, we can pass to 
the limit in the equation (\ref{4}), and  conclude that $u(t)$ solves this equation on $(-\infty,\infty)$. We have $|u(t)| \leq K$ on $(-\infty,\infty)$, proving the existence of a uniformly bounded solution. 

Turning to the uniqueness, the difference $w(t)$ of any two bounded solutions $u(t)$ and $\tilde{u}(t)$ of (\ref{1}) would be a bounded for all $t$ solution of  the linear equation
\beq
\lbl{9}
w''-b(t)w=0 \,,
\eeq
with $b(t)=a(t)(u^2+u\tilde{u}+\tilde{u}^2)>0$. It follows that $w(t)$ is convex when it is positive. If at some $t_0$, $w(t_0)>0$ and $w'(t_0)>0$ ($w'(t_0)<0$), then $w(t)$ is unbounded as $t \ra \infty$ ($t \ra -\infty$), a contradiction. A similar contradiction occurs if $w(t_0)<0$ for some $t_0$. Therefore, $w\equiv 0$.
\end{proof}

\noindent
{\bf Remark 1.} To prove the existence of solutions of (\ref{4}) , we could alternatively consider the corresponding variational functional $J(u): \, H^1_0(-L,L) \ra \R$, defined by
\[
J(u)=\int_{-L}^L \left[\frac{\lt({u'}\rt)^2}{2}+a(t) \frac{u^4}{4}+f(t)u \right] \, dt \,.
\]
Since for any $\epsilon>0$
\[
\lt|\int_{-L}^L f(t)u \, dt\rt| \leq \epsilon \int_{-L}^L u^2 \, dt+c(\epsilon) \int_{-L}^L f^2 \, dt 
\]
\[
\leq \epsilon \int_{-L}^L u^2 \, dt+c_1 \,, \s \mbox{with $c_1=c_1(L,\epsilon)$} \,,
\]
and 
$$\int_{-L}^L u^2 \, dt \leq c_2(L) \int_{-L}^L \lt(u'\rt)^2 \, dt,$$
we see (noting $a(t)u^4\geq 0$) that 
$$J(u) \geq c_3 \int_{-L}^L \lt(u'\rt)^2\, dt-c_4 $$ 
for some $c_3,c_4 >0$, so that $J(u)$  is bounded from below, coercive and convex in $u'$. Hence $J(u)$ has a minimizer in $H^1_0(-L,L)$, which gives us a classical solution of (\ref{4}), see e.g., L. Evans \cite{E}. However, to get a uniform in $L$ estimate of $\int_{-L}^L \lt(u'\rt)^2 \, dt$ (needed to conclude the equicontinuity in (\ref{8})), one would have to assume that $\int_{-\infty}^{\infty} f^2(t) \, dt < \infty$, giving a weaker result than above.
\medskip

We now discuss the dynamical significance of the bounded solution, established in Theorem \ref{thm:1}, let us call it $u_0(t)$. The difference of any two solutions of (\ref{1}) satisfies (\ref{9}).
 We see from (\ref{9}) that any two solutions of (\ref{1}) intersect at most once. Also from (\ref{9}), we can expect $u_0(t)$ to have one-dimensional stable manifold as $t \ra \pm \infty$. It follows that $u_0(t)$ provides the only possible asymptotic form of the solutions that are bounded as $t \ra \infty$ (or $t \ra -\infty$), while all other solutions become unbounded.
\medskip

Next we show that the conditions of this theorem cannot be completely removed. If $a(t) \equiv 0$, then for $f(t)=1$, all solutions of (\ref{1}) are unbounded as $t \ra \pm \infty$. The same situation  may occur in case $a(t)>0$, if $f(t)$ is unbounded. Indeed, the equation 
\beq
\lbl{9.1}
u''-u^3=2 \cos t-t \sin t-t^3 \sin ^3 t
\eeq
has a solution $u(t)=t \sin t$. Let $\tilde u(t)$ be any other solution of (\ref{9.1}). Then $w(t)=u(t)-\tilde u(t)$ satisfies (\ref{9}), with $b(t)=u^2+u\tilde u+\tilde{u}^2>0$.
Clearly, $w(t)$ cannot have points of positive local maximum, or negative local minimum.  But then $\tilde u(t)$ cannot remain bounded as $t \ra \pm \infty$, since in such a case the function $w(t)$ would be unbounded with points of positive local maximum and negative local minimum. It follows that all solutions of (\ref{9.1}) are unbounded as $t \ra \pm \infty$.
\medskip

The  approach of Theorem \ref{thm:1} is applicable to more general equations and systems. For example, we have the following theorem.
\begin{thm}\label{t2.2}
Consider the system (for $u=u(t)$ and $v=v(t)$)
\beqa
\lbl{**}
\begin{cases}
u''-a_1(t)f(u,v)=h_1(t), \\
v''-a_2(t)g(u,v)=h_2(t). 
\end{cases}
\eeqa
Assume that the functions $a_i(t) \in C(\R)$ satisfy $a_0\leq a_i(t)\leq a_1$ for all $t\in\R$ and some constants $0<a_0\leq a_1$, while $h_i(t) \in C(\R)$ are uniformly bounded, $i=1,2$.
Assume that the functions $f(x,y)$ and $g(x,y)$ are continuous on $\R^2$, and 
\begin{equation}\label{2.10} 
f(x,y) \ra \infty  \; (-\infty) \text{ as } x \ra \infty \; (-\infty), \mbox{uniformly in $y$},
\end{equation}
and 
\begin{equation}\label{2.11}
g(x,y) \ra \infty  \; (-\infty) \text{ as } y \ra \infty  \; (-\infty), \mbox{uniformly in $x$}. 
\end{equation}
Assume that
\beq
\lbl{+++}
xf(x,y) \geq \alpha \,, \s \mbox{and} \s yg(x,y) \geq \alpha \,,
\eeq
for some $\alpha \in \R$, and all $(x,y)\in\R^2$.  Assume finally that the quadratic form in $(w,z)$
\begin{equation}
\lbl{15}
a_1(t)f_x(x,y)w^2+\left(a_1(t)f_y(x,y)+a_2(t)g_x(x,y) \right)wz+a_2(t)g_y(x,y)z^2
\end{equation}
is positive semi-definite for all $t$, $x$ and $y$.
Then the problem (\ref{**}) has a classical solution uniformly bounded for all $t \in (-\infty,\infty)$. 
\end{thm}

\begin{proof}
To prove the existence of solutions for the corresponding Dirichlet problem on $(-L,L)$,
\begin{equation}
\lbl{**+}
\begin{cases}
u''-a_1(t)f(u,v)=h_1(t) \s\s \mbox{for $t \in (-L,L)$}, \s u(-L)=u(L)=0, \\
v''-a_2(t)g(u,v)=h_2(t) \s\s \mbox{for $t \in (-L,L)$}, \s v(-L)=v(L)=0,
\end{cases}
\end{equation}
we embed it  into a family of problems
\begin{equation}
\lbl{12}
\begin{cases}
u''-\la a_1(t)f(u,v)=h_1(t) \s\s \mbox{for $t \in (-L,L)$}, \s u(-L)=u(L)=0,\\
v''-\la a_2(t)g(u,v)=h_2(t) \s\s \mbox{for $t \in (-L,L)$}, \s v(-L)=v(L)=0,
\end{cases}
\end{equation}
with $0 \leq \la \leq 1$. The implicit function theorem applies, since the corresponding linearized problem 
\beqa
\begin{cases}
w''-\la a_1(t)\left(f_x(u,v)w+f_y(u,v)z \right)=0 \s\s \mbox{for $t \in (-L,L)$}, \\\nonumber
z''-\la a_2(t)\left(g_x(u,v)w+g_y(u,v)z \right)=0 \s\s \mbox{for $t \in (-L,L)$},  \\ \nonumber
w(-L)=w(L)=z(-L)=z(L)=0 
\end{cases}
\eeqa
has only the trivial solution $w=z=0$. This follows by multiplying the first equation by $w$, the second one by $z$, integrating, adding the results, and using the condition (\ref{15}). Using (\ref{+++}), we obtain a uniform in $\lambda$ bound on the $H^1$ norm of the solution of (\ref{12}), so that the continuation can be performed for all $0 \leq \la \leq 1$. At $\la =1$, we obtain a solution of (\ref{**+}).
\medskip

From the first equation in (\ref{**+}) and the assumption (\ref{2.10}) we conclude the bound (\ref{5}) on $u(t)$, and a similar bound on $v(t)$ follows from the second equation in  (\ref{**+}) and the assumption (\ref{2.11}), the same way as we did for a single equation. Using the equations in (\ref{**+}), we obtain uniform bounds on $u''$ and $v''$, and the uniform bounds on $u'$ and $v'$ follow from (\ref{2.7}). Hence, we have the estimates (\ref{5}) for $u$ and $v$. We then let $L \ra \infty$, and pass to the limit along the diagonal sequence, as in the proof of Theorem \ref{thm:1}, to conclude the proof of Theorem \ref{t2.2}.
\end{proof}

\noindent
{\bf Example 1.} Theorem \ref{t2.2} applies in case $f(x,y)=x+x^{2n+1}+r(y)$, $g(x,y)=y+y^{2m+1}+s(x)$, with positive integers $n$ and $m$,  assuming that  the functions $r(y)$ and $s(x)$ are bounded and have small enough derivatives for all $x$ and $y$, and the functions $a_i(t)$ and $h_i(t)$, $i=1,2$, satisfy the assumptions of the theorem.

\section{Bounded solutions of Hamiltonian  systems}
\setcounter{equation}{0}
\setcounter{thm}{0}
\setcounter{lma}{0}

We use variational approach to get a similar result for a class of Hamiltonian  systems. 
We shall be  looking for uniformly bounded solutions $u\in H^1(\R;\R^m)$ of the system 
\beq
\label{3.1}
u_i''  - a(t) V_{z_i}(u_1,u_2,\ldots,u_m) = f_i(t) \,, \s \s i=1,\ldots,m \,.
\eeq
Here $u_i(t)$ are  the unknown functions,   $a(t)$ and $f_i(t)$ are  given functions on $\R$, $i=1,\ldots,m$, and $V(z)$ is a given function on $\R^m$.

\begin{thm}\label{t3}
Assume that $a(t) \in C(\R)$ satisfies $a_0\leq a(t)\leq a_1$ for all $t$, and some constants $0<a_0\leq a_1$. Assume that $f_i(t) \in C(\R)$, with $|f_i(t)| \leq M$  for some $M>0$ and all $i$ and $t \in \R$. Also assume that $V(z) \in C^1(\R^m)$ satisfies
\beq
\label{3.1a}
\lim _{z_i \ra \infty}V_{z_i}=\infty \;,\lim _{z_i \ra -\infty}V_{z_i}=-\infty \,, \s \mbox{uniformly in all $z_j \ne z_i$} \,,
\eeq
and
\beq
\lbl{3.2}
a(t)V(z) +\sum _{i=1}^m z_i f_i(t) \geq  -f_0(t) \,, \s \mbox{for all $t \in R$, and $z_i \in R$} \,,
\eeq
with some $f_0(t)>0$ satisfying $\int_{-\infty}^{\infty} f_0(t) \, dt<\infty$. Then the system (\ref{3.1}) has a uniformly bounded solution $u_i(t) \in H^1(\R)$, $i=1,\ldots,m$ (i.e., for some constant $K>0$, $|u_i(t)|<K$ for all $t \in \R$, and all $i$).
\end{thm}

\begin{proof}
As in the previous section, we approximate  solution of 
(\ref{3.1}) by solutions of the corresponding Dirichlet problems ($i=1,\ldots,m$)
\beq
u_i''  - a(t) V_{z_i}(u) = f_i(t) \,, \s\; \mbox{for} \; t\in (-L,L), \; 
u(-L) = u(L) = 0 \, ,
\label{3.11}
\eeq
as $L \ra \infty$. Solutions of (\ref{3.11}) can be obtained as  critical points of the corresponding variational functional $J(u): \left[H^1_0(-L,L)\right]^m \ra \R$ defined as
\[
J(u) := \int^{L}_{-L} \left[ \sum_{i=1}^m\left( \frac{1}{2} {u_i'}^2(t)  
+u_i(t)  f_i(t)\right) +a(t) V(u(t))\right] \, dt \,.
\]
By (\ref{3.2}), $J(u) \geq c_1(L)\sum_{i=1}^m ||u_i||_{H^1(-L,L)}-c_2$, for some positive constants $c_1$ and $c_2$, so that $J(u)$ is bounded from below, coercive and convex in $u'$. Hence, $J(u)$ has a minimizer in $\left[H^1_0(-L,L)\right]^m$, giving us a classical solution of (\ref{3.11}), see e.g., L. Evans \cite{E}.
\medskip

We now take a sequence $L_j \rightarrow \infty\,$, and 
denote by $u_j(t) \in H^1(\R;\R^m)$  a  vector solution of the problem (\ref{3.11}) on the interval $(-L_j,L_j)$, extended as zero vector to the outside of the interval $(-L_j,L_j)$. By our condition (\ref{3.1a}), we conclude a component-wise bound of  $|u_j(t)|$, uniformly in $j$ and $t$. The crucial observation (originated from \cite{KL}) is that the variational method provides a uniform in $j$ bound on $\lVert u'_j(t)\rVert_{L^2(-\infty,\infty)}$. Indeed, we have $H^1_0(-L,L) \subset H^1_0(-\tilde L,\tilde L)$ for $\tilde L>L$. If we now denote by $M_L$ the minimum value of $J(u)$ on $\left[H^1_0(-L,L)\right]^m$, then $M_L$ is non-increasing in $L$ (there are more competing functions for larger $L$), and in particular $J(u_j) \leq M_1$ if $L_j>1$. In view of the condition (\ref{3.2}), this provides us with a uniform in $j$ bound on $\int^{L_j}_{-L_j}  \sum_{i=1}^m \lt(u_{j,i}'(t)\rt)^2 \,dt$, from which  we conclude  that the sequence $\{u_j(t)\}$ is equicontinuous on every bounded interval (as in (\ref{8}) above). With the sequence $\{u_j(t)\}$  equicontinuous and uniformly bounded on every interval $[-L_p,L_p]\,$, it converges uniformly to some $u\in C(\R;\R^m)$ on $[-L_p,L_p]\,$. From the equation (\ref{3.11}), we have uniform convergence of $\{u_j''\}$, and hence uniform convergence of $\{u_j'\}$ follows from (\ref{2.7}). We complete the proof as in the proof of Theorem \ref{thm:1}.
\end{proof}

\noindent
{\bf Example 2.} Consider the case $m=2$, $V(z)=z_1^4+z_2^2+h(z_1,z_2)$, with $h(z_1,z_2)>0$ and $h_{z_1}(z)$, $h_{z_2}(z)$ bounded on $\R^2$. We consider the system
\begin{equation*}
\begin{cases}
u_1''-a(t) \left(4u_1^3+h_{z_1}(u) \right)=f_1(t),\\
u_2''-a(t) \left(2u_2+h_{z_2}(u) \right)=f_2(t),
\end{cases}
\end{equation*}
where the functions $a(t), f_1(t), f_2(t)$ satisfy the assumptions of Theorem \ref{t3}. Applying Young's inequality, we obtain
$$\lt|u_1(t) f_1(t)\rt| \leq \epsilon u_1^4(t)+c_1(\epsilon)   f_1^{4/3}(t), $$ 
and 
$$\lt|u_2(t) f_2(t)\rt| \leq \epsilon u_2^2(t)+c_2(\epsilon)  f_2^{2}(t). $$
Therefore, we get for some $c_3>0$
\begin{equation*}
a(t) \left(u_1^4+u_2^2+h(u_1,u_2) \right)+u_1(t) f_1(t)+u_2(t) f_2(t) \geq -c_3 \left( f_1^{4/3}(t)+f_2^{2}(t) \right).
\end{equation*}
Hence, Theorem \ref{t3} applies provided that $\int_{-\infty}^{\infty} \left( f_1^{4/3}(t)+f_2^{2}(t) \right)\, dt<\infty$.

\section{Bounded solutions of Hamiltonian PDE systems}
\setcounter{equation}{0}
\setcounter{thm}{0}
\setcounter{lma}{0}

In this section, we use a combination of the variational approach and elliptic estimates to show that similar results can be obtained for Hamiltonian PDE systems. We shall be  looking for uniformly bounded solutions $u=\lt(u_1,...,u_m\rt)\in H^1(\R^n;\R^m)$, for $n>1$, of the system 
\beq
\label{4.1}
\De u_i  - a(x) V_{z_i}(u) = f_i(x) \,, \s \s i=1,\ldots,m \,.
\eeq
Here $u_i(x)$ are  the unknown functions,   $a(x)$ and $f_i(x)$ are  given functions on $\mathbb{R}^n$, $i=1,\ldots,m$, and $V(z)$ is  a given function on $\R^m$. We shall denote the gradient of $a(x)$ by $Da(x)$.

\begin{thm}\label{t4.1}
Assume that $a(x), f_i(x) \in C^{\infty}(\R^n)$ and $V(z) \in C^\infty(\R^m)$. In addition, assume that there exist constants $0<a_0\leq a_1$ and $M>0$ such that $a_0\leq a(x) \leq a_1$ and $|f_i(x)|, |Da(x)|, |Df_i(x)| \leq M$ for all $x \in \R^n$ and $i=1,...,m$. Assume also that
\beq
\label{4.1a}
\lim _{z_i \ra \infty}V_{z_i}=\infty \;,\lim _{z_i \ra -\infty}V_{z_i}=-\infty \,, \s \mbox{uniformly in all $z_j \ne z_i$} \,,
\eeq
and
\beq
\lbl{4.2}
a(x)V(z) +\sum _{i=1}^m z_i f_i(x) \geq  -f_0(x) \,,
\eeq
for all $x\in\R^n$, $z\in \R^m$ and some function $f_0(x)>0$ satisfying $\int_{\R^n} f_0(x) \, dx<\infty$. Then the system (\ref{4.1}) has a uniformly bounded classical solution $u(x)$, with $u_i(x) \in C^2(\R^n)$, $i=1,\ldots,m$.
\end{thm}

As in the proof of Theorem \ref{t3}, we approximate solutions of the system (\ref{4.1}) by solutions of the following system
\begin{equation}\label{4.3}
\begin{cases}
\De u_i(x)  - a(x) V_{z_i}\lt(u(x)\rt) = f_i(x) \quad\mbox{ for } x\in B_{L}(0), \; \\
u_i(x) = 0  \quad\text{ for } x\in\partial B_L(0) \,,
\end{cases}
\end{equation}
where $B_{L}(0)=\{ x \in R^n \, : \, |x|<L \}$.

\begin{lma}\label{l4.1}
Assume that $a(x), f_i(x) \in C^{\infty}(\R^n)$ and $V(z) \in C^\infty(\R^m)$. In addition, assume that the condition (\ref{4.2}) is satisfied. Then the system (\ref{4.3}) has a classical solution $u_L=\lt(u_{L,1},...,u_{L,m}\rt)\in C^{2}(\overline{B_L(0)};\R^m)$.
\end{lma}

\begin{proof}
We consider the following variational approach: the functional 
\begin{equation*}
J(u) := \int_{B_L(0)} \left[ \sum_{i=1}^m\left( \frac{1}{2} \lt|\nabla u_i\rt|^2  
+u_i(x)  f_i(x)\right) +a(x) V(u(x))\right] \, dx \,
\end{equation*}
is minimized over $H^1_0(B_L(0);\R^m)$. From the condition (\ref{4.2}), we have 
\begin{equation*}
J(u)\geq c_1(L)\lVert u\rVert_{H^1(B_L(0);\R^m)}^2-c_2
\end{equation*}
for some positive constants $c_1, c_2$. Therefore, $J$ is bounded below, coercive and convex in $\nabla u$. Hence, it has a minimizer $u_L\in H^1_0(B_L(0);\R^m)$ that satisfies the system (\ref{4.3}). (See Theorem 2 in Section 8.2.2 of \cite{E}.) Now $u_L$ solves the following elliptic system
\begin{equation*} %\label{4.4}
\begin{cases}
\De u_{L,i}  = a(x) V_{z_i}(u_L) + f_i(x) \quad\text{ in } B_L(0),\\
u_{L,i} = 0  \quad\text{ on } \partial B_L(0).
\end{cases}
\end{equation*}
For any $i$, since $a, f_i$ and  $V$ are all smooth and $u_L\in H^1_0$, it follows from standard elliptic estimates that $u_{L,i}\in H^3(B_L(0))$, and therefore $u_L\in H^3(B_L(0);\R^m)$. (See Theorem 8.13 in \cite{GT}.) By a bootstrapping argument and the Sobolev embedding theorem, one has $u_{L,i}\in C^{2}(\overline{B_L(0)})$ for all $i$ and hence $u_L$ is a classical solution to (\ref{4.3}).
\end{proof}

In the next lemma, we apply interior estimates for classical solutions of the Poisson equation to the function $u_L$ found in Lemma \ref{l4.1}. We introduce some notations from \cite{GT}. Let $\Omega\in\R^n$ be a bounded domain and $u\in C^{2,\alpha}(\Omega)$ for some $0<\alpha<1$. We set
\begin{equation*}
|D^k u|_{0;\Omega} := \sup_{|\beta|=k}\sup_{\Omega}|D^{\beta}u|, \quad k=0,1,2,
\end{equation*}
and
\begin{equation*}
[D^k u]_{\alpha;\Omega} := \sup_{|\beta|=k}\sup_{x,y\in\Omega,x\ne y}\frac{\lt|D^{\beta}u(x)-D^{\beta}u(y)\rt|}{|x-y|^{\alpha}}.
\end{equation*}

\begin{lma}\label{l4.2}
Given $L>2$ and $0<\alpha<1$, under the assumptions of Theorem \ref{t4.1}, there exists a constant $K$ independent of $L$ such that the function $u_L$ found in Lemma \ref{l4.1} satisfies
\beq\label{4.14}
|u_{L}|_{0;\overline{B_{L}(0)}}, |Du_{L}|_{0;\overline{B_{L'}(0)}},|D^2u_{L}|_{0;\overline{B_{L''}(0)}},[D^2u_{L}]_{\alpha;\overline{B_{L''}(0)}} \leq K,
\eeq
where $L'=L-1$ and $L''=L-2$.
\end{lma}

\begin{proof}
We fix an arbitrary index $i\in\{1,...,m\}$, and omit the subscript $L$. Therefore, we denote $u=u_L$ and $u_i=u_{L,i}$. Suppose $x_0\in B_L(0)$ is such that $u_i(x_0)$ is a positive maximum of $u_i$. Then since $\De u_i(x_0)\leq 0$, it follows from (\ref{4.3}) that 
\begin{equation*}
a(x_0) V_{z_i}(u(x_0)) + f_i(x_0) \leq 0
\end{equation*}
and hence 
\begin{equation}\label{4.5}
V_{z_i}(u(x_0)) \leq \frac{M}{a_0}.
\end{equation}
The assumption (\ref{4.1a}) and (\ref{4.5}) then guarantee that $u_i(x_0)$ is bounded from above independent of $L$. Similarly, we have the minimum of $u_i$ is bounded from below independent of $L$. Since this holds for all $i$, we deduce
\beq\label{4.10}
|u|_{0;\overline{B_{L}(0)}} \leq K_0
\eeq
for some $K_0$ independent of $L$. 

We denote $F_i(u,x):= a(x) V_{z_i}(u) + f_i(x)$. It follows from Lemma \ref{l4.1} and (\ref{4.10}) that $F_i\in C^2(\overline{B_L(0)})$ and $|F_i(u,x)|_{0;\overline{B_L(0)}}$ is bounded independent of $L$. Let $\bar{x}\in \overline{B_{L'}(0)}$ and $w$ be the Newtonian potential of $F_i$ on $B_{1}(\bar{x})$, then it is clear that $u_i=w+v$ for some harmonic function $v$ on $B_{1}(\bar{x})$. For all $x\in B_{1}(\bar{x})$ we have
\begin{equation*}
w(x)=\int_{B_{1}(\bar{x})}\Gamma(x-y) F_i(u(y),y)dy
\end{equation*}
and 
\begin{equation*}
Dw(x)=\int_{B_{1}(\bar{x})}D\Gamma(x-y) F_i(u(y),y)dy,
\end{equation*}
where $\Gamma$ is the fundamental solution of the Laplacian in $\R^n$ (see \cite{GT} Lemma 4.1). Using properties of $\Gamma$ and uniform boundedness of $F_i$, it is easy to check that
\begin{equation}\label{4.9}
|w|_{0;B_{1}(\bar{x})} \leq C|F_i|_{0;B_{1}(\bar{x})} \text{ and } |Dw|_{0;B_{1}(\bar{x})} \leq C|F_i|_{0;B_{1}(\bar{x})}
\end{equation}
for some constant $C$ depending only on $n$. Therefore we have
\begin{equation}\label{4.6}
|v|_{0;B_{1}(\bar{x})} \leq |u_i|_{0;B_{1}(\bar{x})} + |w|_{0;B_{1}(\bar{x})} \leq C\lt(|u_i|_{0;B_{1}(\bar{x})}+|F_i|_{0;B_{1}(\bar{x})}\rt).
\end{equation}
Using interior estimates for harmonic functions (see \cite{GT} Theorem 2.10), we have
\begin{equation}\label{4.7}
|Dv|_{0;B_{\frac{1}{2}}(\bar{x})} \leq C |v|_{0;B_{1}(\bar{x})}
\end{equation}
for some constant $C$ depending only on $n$, since for any $x\in B_{\frac{1}{2}}(\bar{x})$, we have $dist(x,\partial B_{1}(\bar{x})) \geq \frac{1}{2}$. Now combining (\ref{4.6})-(\ref{4.7}) we obtain
\begin{equation*} %\label{4.8}
|Dv|_{0;B_{\frac{1}{2}}(\bar{x})} \leq C\lt(|u_i|_{0;B_{1}(\bar{x})}+|F_i|_{0;B_{1}(\bar{x})}\rt)
\end{equation*}
for some constant $C$ depending only on $n$. This along with (\ref{4.9}) yields
\begin{equation*}
|Du_i|_{0;B_{\frac{1}{2}}(\bar{x})} \leq C\lt(|u_i|_{0;B_{1}(\bar{x})}+|F_i|_{0;B_{1}(\bar{x})}\rt) \leq C\lt(|u_i|_{0;B_{L}(0)}+|F_i|_{0;B_{L}(0)}\rt)
\end{equation*}
for some constant $C$ depending only on $n$. Now since $\bar{x}$ is arbitrary in $\overline{B_{L'}(0)}$, it follows that
\begin{equation*}
|Du_i|_{0;\overline{B_{L'}(0)}} \leq C\lt(|u_i|_{0;B_{L}(0)}+|F_i|_{0;B_{L}(0)}\rt).
\end{equation*}
In particular, since $|u_i|_{0;B_L(0)}$ and $|F_i|_{0;B_L(0)}$ are bounded independent of $L$, we obtain a uniform bound on $|Du_i|_{0;\overline{B_{L'}(0)}}$ independent of $L$. Hence we have
\beq\label{4.11}
|Du|_{0;\overline{B_{L'}(0)}} \leq K_1
\eeq
for some $K_1$ independent of $L$.

By assumption, both $|Da|_{0;\R^n}$ and $|Df_i|_{0;\R^n}$ are bounded. Since $V$ is smooth, and both $|u|_{0;\overline{B_L'(0)}}$ and $|Du|_{0;\overline{B_L'(0)}}$ are bounded independent of $L$, it is clear that $|DF_i|_{0;\overline{B_{L'}(0)}}$ is bounded independent of $L$. It follows that $[F_i]_{\alpha;\overline{B_{L'}(0)}}$ is bounded independent of $L$. For all $\bar{x}\in \overline{B_{L''}(0)}$ we deduce from \cite{GT} Theorem 4.6 that
\begin{equation*}
\begin{split}
&\lt(\frac{1}{3}\rt)^2|D^2u_i|_{0;B_{\frac{1}{3}}(\bar{x})} + \lt(\frac{1}{3}\rt)^{2+\alpha}[D^2u_i]_{\alpha;B_{\frac{1}{3}}(\bar{x})}\\
&\leq C\lt[|u_i|_{0;B_{\frac{2}{3}}(\bar{x})}+\lt(\frac{1}{3}\rt)^2\lt(|F_i|_{0;B_{\frac{2}{3}}(\bar{x})}+\lt(\frac{2}{3}\rt)^{\alpha}[F_i]_{\alpha;B_{\frac{2}{3}}(\bar{x})}\rt)\rt]\\
&\leq C\lt(|u_i|_{0;B_{L'}(0)}+|F_i|_{0;B_{L'}(0)}+[F_i]_{\alpha;B_{L'}(0)}\rt)
\end{split}
\end{equation*}
for some constant $C$ depending only on $n$ and $\alpha$. Since $\bar{x}\in \overline{B_{L''}(0)}$ is arbitrary and the above right hand side is bounded independent of $L$, we conclude that
\beq\label{4.12}
|D^2u|_{0;\overline{B_{L''}(0)}}, \quad [D^2u]_{\alpha;\overline{B_{L''}(0)}} \leq K_2
\eeq
for some $K_2$ independent of $L$. Putting (\ref{4.10}), (\ref{4.11}), (\ref{4.12}) together and setting $K:=\max\{K_1,K_2,K_3\}$, we obtain (\ref{4.14}). 
\end{proof}

\begin{proof}[Proof of Theorem \ref{t4.1}]
We take an increasing sequence $\{L_j\}_j$ with $L_1>2$ and $\lim_{j\rightarrow \infty}L_j=\infty$, and denote by $u_j=u_{L_j}$ the function found in Lemma \ref{l4.1}. We extend $u_j$ to be zero outside $\overline{B_{L_j}(0)}$. Note that $u_j\in C^{2,\alpha}(\overline{B_{L_j}(0)};\R^m)$ but does not need to be smooth on $\R^n$. On each $\overline{B_{L_{p}''}(0)}$, it follows from Lemma \ref{l4.2} that the sequences $\{u_j\}_{j\geq p}$, $\{Du_j\}_{j\geq p}$ and $\{D^2u_j\}_{j\geq p}$ are all uniformly bounded and equicontinuous. Using the diagonal arguments as in the proof of Theorem \ref{thm:1}, one can find a subsequence $\{u_{j_k}\}$ such that $\{u_{j_k}\}$, $\{Du_{j_k}\}$ and $\{D^2u_{j_k}\}$ are all uniformly convergent on all $\overline{B_{L_p''}(0)}$. In particular, there exists $u\in C(\R^n;\R^m)$ such that
\begin{equation}\label{4.15}
u_{j_k}\rightarrow u \quad\text{ uniformly on all bounded domains in } \R^n.
\end{equation}

It is clear from Lemma \ref{l4.2} that $u$ is bounded on $\R^n$. It remains to show that the vector valued function $u$ satisfies the system (\ref{4.1}). Let $\Omega\subset\R^n$ be any bounded convex domain and $i\in\{1,...,m\}$ be any index. Note that $u_{j_k,i}\in C^2(\overline{\Omega})$ for all $k$ sufficiently large, and there exist $v\in C(\overline{\Omega};\R^n)$ and $w\in C(\overline{\Omega};\R^{n\times n})$ such that
\begin{equation}\label{4.16}
\nabla u_{j_k,i}\rightarrow v \quad \text{and} \quad \nabla^2 u_{j_k,i}\rightarrow w \quad \text{uniformly on } \overline{\Omega},
\end{equation}
where $\nabla^2 u_{j_k,i}$ is the Hessian matrix of $u_{j_k,i}$. Fix $x_0\in\Omega$. For any $x\in \Omega$, we have
\begin{equation*}
u_{j_k,i}(x) = u_{j_k,i}(x_0) + \int_{l_{x_0}^x} \nabla u_{j_k,i}(s)\cdot \tau ds,
\end{equation*}
where $l_{x_0}^x$ is the line segment joining $x_0$ and $x$ and $\tau$ is the unit tangent vector of $l_{x_0}^x$. Using (\ref{4.15}) and (\ref{4.16}), we obtain
\begin{equation*}
u_{i}(x) = u_{i}(x_0) + \int_{l_{x_0}^x} v(s)\cdot \tau ds,
\end{equation*}
and therefore $u_i\in C^1(\Omega)$ and $\nabla u_i = v$. Using similar arguments and (\ref{4.16}), we obtain that $v\in C^1(\Omega)$ and $\nabla v  = w$, and hence $u_i\in C^2(\Omega)$ and $\nabla^2 u_i = w$ in $\Omega$. For $k$ sufficiently large, we know $u_{j_k,i}$ solves
\begin{equation*}
\De u_{j_k,i}  - a(x) V_{z_i}(u_{j_k}) = f_i(x) \,, \s\; \mbox{for} \; x\in \Omega.
\end{equation*}
Passing to the limit as $k\rightarrow \infty$, we have
\begin{equation*}
\De u_{i}  - a(x) V_{z_i}(u) = f_i(x) \,, \s\; \mbox{for} \; x\in \Omega.
\end{equation*}
Since this holds for all bounded convex domains $\Omega\in\R^n$, we conclude that $u\in C^2(\R^n;\R^m)$ is a bounded solution of the system (\ref{4.1}).
\end{proof}

\noindent
{\bf Remark 2.} We can apply Theorem \ref{t4.1} to the system given in Example 2, but with smooth $h$ and the functions $a(x), f_1(x), f_2(x)$ satisfying the additional assumptions in Theorem \ref{t4.1}.

\end{document}